\documentclass[11pt]{amsart}
\usepackage{amssymb,amsmath,txfonts}
\usepackage{hyperref}
\newtheorem{theorem}{Theorem}
\newtheorem{prop}{Proposition}
\newtheorem*{lemma}{Lemma}
\newtheorem*{remark}{Remark}
\newtheorem*{question}{Question}
\newtheorem{claim}{Claim}

\newtheorem{cor}{Corollary}
\newtheorem*{conj}{Conjecture}

\def\XXint#1#2#3{{\setbox0=\hbox{$#1{#2#3}{\int}$}
  \vcenter{\hbox{$#2#3$}}\kern-.5\wd0}}

\author{Gang Liu}
\address{Department of Mathematics\\University of California, Berkeley\\Berkeley, CA 94720}
\email{gangliu@math.berkeley.edu}
\title[Compact K\"ahler manifolds with nonpositive bisectional curvature]{\bf Compact K\"ahler manifolds with nonpositive bisectional curvature}
\date{}

\begin{document}
\maketitle
\begin{abstract}
Let $(M^n, g)$ be a compact K\"ahler manifold with nonpositive bisectional curvature. We show that a finite cover is biholomorphic and isometric to a flat torus bundle
over a compact K\"ahler manifold $N^k$ with $c_1 < 0$.  This confirms a conjecture of Yau.  As a corollary, for any compact K\"ahler manifold with nonpositive bisectional curvature, the Kodaira dimension is equal to the maximal rank of the Ricci tensor.
We also prove a global splitting result under the assumption of certain immersed complex submanifolds.
\end{abstract}
\section{\bf{Introduction}}

The uniformization theorem of Riemann surfaces says the sign of curvature could determine the conformal structure in some sense. Explicitly, if the curvature is positive, it is covered by $\mathbb{P}^1$ or $\mathbb{C}$.
On the other hand, if the curvature is less than a negative constant, it is covered by the unit disk $\mathbb{D}^2$.

It is natural to wonder whether there are generalizations in higher dimensions. For the compact case, the famous Frankel conjecture says if
a compact K\"ahler manifold has positive holomorphic bisectional curvature, then it is biholomorphic to $\mathbb{CP}^n$. This conjecture was solved
by Mori \cite{[Mo]} and Siu-Yau \cite{[SiY]} independently. In fact Mori proved the stronger Hartshorne conjecture. Later, Mok \cite{[Mok]} solved the generalized
Frankel conjecture. The result says that, if a compact K\"ahler manifold has nonnegative holomorphic bisectional curvature, then the universal cover is isometric-biholomorphic to $(\mathbb{C}^k, g_0)\times (\mathbb{P}^{n_1}, \theta_1)\times\cdot\cdot\cdot\times(\mathbb{P}^{n_l}, \theta_l)\times(M_1, g_1)\times\cdot\cdot\cdot\times(M_i, g_i)$,
where $g_0$ is flat; $\theta_k$ are metrics on $\mathbb{P}^{n_k}$ with nonnegative holomorphic bisectional curvature; $(M_j, g_j)$ are compact irreducible Hermitian symmetric spaces.

If the curvature is negative, the current knowledge is much less satisfactory. For example,
 a famous conjecture of Yau says if a complete simply connected K\"ahler manifold has sectional curvature between
 two negative constants, then it is a bounded domain in $\mathbb{C}^n$. So far, it is not even known whether there exists a nontrivial bounded holomorphic
 function on such manifolds.

As in the Riemannian case, it is often important to understand the difference between the negatively curved case and the nonpositive case. The former tends to be hyperbolic in some sense, while the latter usually possesses some rigidity properties. For compact K\"ahler manifolds with nonpositive holomorphic bisectional curvature, there is a conjecture of Yau (page 2 of \cite{[WZ]}, also \cite{[WZ1]}):
\begin{conj}
Let $M^n$ be a compact K\"ahler manifold with nonpositive holomorphic bisectional curvature. Then there exists a finite cover $M'$ of $M$ such that $M'$ is a holomorphic and metric fibre bundle over a compact K\"ahler manifold $N$ with nonpositive bisectional curvature and $c_1(N)<0$, and the fiber is a flat complex torus.
\end{conj}
Recall that a fiber bundle is called a metric bundle, if for any
$p\in N$, there is some neighborhood $p\in U \subset N$ such that the bundle
over $U$ is isometric to the product of the fiber with $U$.
In \cite{[Y]}, Yau proved the following
\begin{theorem}
Let $M$ be a compact complex submanifold of a complex torus $T^n$. Then $M$ is a torus bundle over a complex submanifold $N$ in $T^n$ such that the induced K\"ahler metric on $N$ has negative definite Ricci tensor in an open dense set of $N$.
\end{theorem}
Since complex submanifolds in $T^n$ has nonpositive holomorphic bisectional curvature, Yau's theorem confirms the conjecture when $M$ is a complex submanifold of $T^n$.
Zheng \cite{[Z]} proved this conjecture under the extra assumption that $M$ has nonpositive sectional curvature and the metric is real analytic.
In \cite{[WZ]}, Wu and Zheng proved this conjecture by only assuming that the metric is real analytic. They first proved a local splitting result by a careful study of the foliation at the points where the Ricci tensor has maximal rank. By real analyticity, the foliation could be extended to the whole manifold.  
In this note we confirm the conjecture above.
\begin{theorem}\label{thm1}
Let $(M^n, g)$ be a compact K\"ahler manifold with nonpositive holomorphic bisectional curvature. Then there exists a finite cover $M'$ of $M$ such that $M'$ is a holomorphic and metric fiber bundle over a compact K\"ahler manifold $N^k$ with nonpositive bisectional curvature and the Ricci curvature is strict negative in an open set on $N$. Thus $c_1(N)<0$. The fiber is a flat complex torus $T$. Furthermore, $M'$ is diffeomorphic to $T\times N$. Finally, if $r$ is the maximal rank of the 
Ricci curvature of $g$, then $r = k=Kod(M)$ where $Kod(M)$ is the Kodaira dimension of $M$. 
\end{theorem}

\begin{cor}
For any compact K\"ahler manifold with nonpositive bisectional curvature, the Kodaira dimension is equal to the maximal rank of the Ricci tensor.
\end{cor}
\begin{cor}\label{cor-1}
 Let $M^n$ be a compact K\"ahler manifold with nonpositive bisectional curvature. If the Ricci tensor degenerates everywhere, i.e., the maximal rank of the Ricci tensor is strictly less than $n$, then the universal cover splits off a nontrivial complex Euclidean factor holomorphically and isometrically.
\end{cor}

One can ask a question similar to corollary \ref{cor-1} in the Riemannian setting. Namely, for a compact Riemannian manifold with nonpositive sectional curvature, if the Ricci tensor degenerates everywhere,  is it true that the universal cover has a nontrivial Euclidean factor? In Guler and Zheng's paper \cite{[GZ]}, a counterexample (due to Gromov) is given. We explain the example in some details (Page 2 in \cite{[GZ]}) for comparison with the K\"ahler case. 

Take a punctured torus $Y$ and equip it with nonpositive curvature such that the metric near the boundary is isometric to $[0, 1]\times \mathbb{S}^1$.  Take two copies of $Y\times \mathbb{S}^1$ and glue them along the boundary, but with the $\mathbb{S}^1$ factors switched. Then the resulting 3-manifold $M$ has nonpositive sectional curvature with Ricci tensor degenerates everywhere,
but the universal cover does not contain an Euclidean factor. 

Similarly a four manifold is obtained if we glue two copies of $Y\times \mathbb{S}^1\times\mathbb{S}^1$ by switching some $\mathbb{S}^1$ factors.
By theorem \ref{thm1}, this cannot be a counterexample in the K\"ahler case.  The reason is that if we switch the $\mathbb{S}^1$ factor, the metrics match, but the complex structures do not match!

It is interesting to compare the Riemannian case with the K\"ahler case. Here we are assuming that $M$ is a compact (K\"ahler) manifold with nonpositive sectional (bisectional) curvature and the Ricci tensor degenerates (If $M$ is only Riemannian, just ignore the parentheses). 
\begin{itemize}
\item
If $M$ is Riemannian, the universal cover of $M$ does \textbf{not} necessarily have an Euclidean factor.
\item
If $M$ is Riemannian with real analytic metric,  the universal cover of $M$ does have a nontrivial Euclidean factor. 
\item
If $M$ is K\"ahler with real analytic metric, the universal cover of $M$ has a nontrivial complex Euclidean factor (in metric and holomorphic sense). 
\item
If $M$ is K\"ahler, the universal cover of $M$ has a nontrivial complex Euclidean factor (in metric and holomorphic sense).
\end{itemize}
Among the four conclusions, the first two are given in \cite{[GZ]};  the third is proved in \cite{[WZ]}; the last one is corollary \ref{cor-1}.

\begin{theorem}\label{thm5}
Let $(M^n, g)$ be a compact K\"ahler manifold with nonpositive holomorphic bisectional curvature. Suppose $N^{n-k} \subset M$ is a complete (compact or noncompact) immersed complex submanifold of $M$ which is flat and totally geodesic.  If in addition, $Ric(M)|_{TN} = 0$, then $M$ splits globally, i.e., the universal cover $\tilde{M}$ is isometric and biholomorphic to $\mathbb{C}^{n-k} \times Y^{k}$ where $Y^{k}$ is a complete K\"ahler manifold of dimension $k$.
\end{theorem}
\begin{remark}
 All assumptions in theorem \ref{thm5} are ``local" around $N$, except that the holomorphic bisectional curvature on $M$ is nonpositive. Thus it might be interesting to see that local assumptions imply global splitting.
Theorem \ref{thm5} also holds if we assume the manifold has nonnegative bisectional curvature. We can also weaken the conditions by assuming that $M$ is complete with bounded curvature. Finally, note that theorem \ref{thm5} is not true for the Riemannian case. \end{remark}

In \cite{[WZ]}, Wu and Zheng studied the foliation given by the kernel of the Ricci tensor at the points where the Ricci tensor has the maximal rank. See section $2$, part $3$ for some explanations of this foliation.  For $0\leq i\leq n$, define $U(i) = \{x\in M| rank(Ric(x)) = i\}$. If $p$ is an interior point of $U(i)$, then there is a foliation near $p$ given by the kernel of the Ricci tensor. We can extend the leaves as long as the points are in the interior of $U(i)$. 
It is natural to wonder the following:
\begin{question} Will the leaf through $p$ touch the boundary of $U(i)$? \end{question}

In \cite{[F]},  Ferus showed that if $i$ is the maximal rank of the Ricci tensor, then the leaf through $p$ will stay in $U(i)$.  Thus it is complete.
We have a complete answer to the question above:
\begin{cor}\label{cor3}
The leaf through $p$ does not touch the boundary of $U(i)$ if and only if $i$ the maximal rank of the Ricci tensor on $M$.
\end{cor}

Next we discuss two applications of theorem \ref{thm1}.
The existence of canonical metrics is a central topic in K\"ahler geometry. We shall restrict to the case when $c_1\leq 0$. Yau \cite{[Y1]} solved the famous Calabi conjecture. He proved that any compact K\"ahler manifold with $c_1 < 0$ or $c_1 = 0$ admits a K\"ahler-Einstein metric.  Aubin \cite{[Au]} also obtained the proof when $c_1 < 0$.  It is natural to ask whether there exist canonical metrics on K\"ahler manifolds with nonpositive bisectional curvature. 
\begin{cor}\label{cor1}
Let $(M^n, g_0)$ be a compact K\"ahler manifold with nonpositive holomorphic bisectional curvature. Then the manifold admits a canonical metric $g$ which is locally a product of a flat metric with a K\"ahler-Einstein metric with negative scalar curvature. More precisely, $(M, g)$ is locally biholomorphic and isometric to $(D^{n-k}, g_1) \times (U^k, g_2)$. Here $k = Kod(M)$; $(D^{n-k}, g_1)$ is a flat complex Euclidean ball with small radius; $(U^k, g_2)$ is a small ball in $\mathbb{C}^k$ such that $Ric(g_2)=-g_2$.
\end{cor}
\begin{proof}
According to theorem \ref{thm1}, there exists a flat fibration $T^{n-k}\to M'\to N$ where $M'$ is a finite cover of $M$.  The universal cover $\tilde{M}$ is biholomorphic to $\mathbb{C}^{n-k}\times \tilde{N}$ where $\tilde{N}\to N$ is the universal covering.  Since $c_1(N) < 0$, $N$ admits a unique K\"ahler-Einstein metric $g_2$.  Thus $\tilde{N}$ admits a complete K\"ahler-Einstein metric with negative scalar curvature. Any element $a\in \pi_1(M)$ induces a deck transformation $f$ on $\tilde{M}$ which descends to a biholomorphism of $\tilde{N}$. By Yau's Schwarz lemma \cite{[Y2]}, the K\"ahler-Einstein metric on $\tilde{N}$ is unique. Thus $f$ preserves the K\"ahler-Einstein metric $g_2$ on $\tilde{N}$. Therefore, the product metric on $\mathbb{C}^{n-k}\times (\tilde{N}, g_2)$ descends to a metric on $M$ which is canonical. 
\end{proof}

It is also interesting to analyze the long time behavior of the normalized K\"ahler-Ricci flow
\begin{equation}
\frac{\partial g_{i\overline{j}}}{\partial t} = -R_{i\overline{j}}-g_{i\overline{j}}
\end{equation}
on such manifolds. Cao \cite{[C]} proved that if a compact K\"ahler manifold $(M, \omega)$ has $c_1 < 0$ or $c_1 = 0$ (without normalization), then the K\"ahler-Ricci flow converges to a K\"ahler-Einstein metric.  Tsuji \cite{[Ts]} and Tian-Zhang \cite{[TZ]} proved that if a K\"ahler manifold has $c_1\leq 0$, then the normalized K\"ahler-Ricci flow has long time existence.   In \cite{[ST1]}, Song and Tian considered the normalized
K\"ahler-Ricci flow on an elliptic surface $f: X\to\Sigma$ where some of the fibers may be singular. 
It was shown that the flow converges to a generalized K\"ahler-Einstein metric. This result was generalized in \cite{[ST2]} to the fibration $f: X \to X_{can}$ where $X$ is a nonsingular
algebraic variety with semi-ample canonical bundle and $X_{can}$ is its canonical model. We have a result in the similar spirit.

\begin{cor}
Let $M^n$ be a compact K\"ahler manifold with nonpositive bisectional curvature. Then for any initial K\"ahler metric  $g(0)$, the normalized K\"ahler-Ricci flow converges in $C^{\infty}(M)$ to the K\"ahler-Einstein metric factor in corollary \ref {cor1}.
\end{cor}
\begin{proof}
Let $M'$ be in theorem \ref{thm1}. We consider the normalized K\"ahler-Ricci flow on $M'$ which is diffeomorphic to $T\times N$. 
Recall a theorem of M. Gill \cite{[G]} which generalizes a theorem in \cite{[SW]} by Song and Weinkove,
\begin{theorem}\label{thm3}
Let $X = Y\times T$ where $Y$ is a compact K\"ahler manifold with negative first Chern class and $T$ is a complex torus. Let $\omega(t)$ be the normalized K\"ahler-Ricci flow on $X$ with any initial metric $\omega(0)$, then 
$\omega(t)$ converges to $\pi^*(\omega_Y)$ in $C^{\infty}(X, \omega_0)$ sense as $t\to \infty$ where $\pi: X\to Y$ is the projection and $\omega_Y$ is the K\"ahler-Einstein metric on $Y$.
\end{theorem}

Note that  $M'$ is not necessarily biholomorphic to $T\times N$.  However, $M'$ is locally biholomorphic to $T\times U$ where $U$ is an open set in $N$. Thus there is a  flat metric $\omega_T$ on the fiber independent of the projection to $N$. Then one can check that the proof of theorem \ref{thm3} in \cite{[G]} works in this case without any modification.
The projection of the K\"ahler-Ricci flow from $M'$ to $M$ concludes the proof.
\end{proof}

 The proof of theorem \ref{thm1} uses Hamilton's Ricci flow \cite{[H1]} and Hamilton's maximum principle for tensors \cite{[H2]}\cite{[CL]}\cite{[BW]}, together with some argument in \cite{[WZ]} by Wu and Zheng. We will use the invariant convex set constructed in \cite{[BW]} by B\"{o}hm and Wilking. The key point is to prove that there exists a small $\epsilon > 0$ such that along the Ricci flow, $Ric(g(t)) \leq 0$ for all $0 < t < \epsilon$ (note that the holomorphic bisectional curvature is not necessarily nonpositive for small $t$). The final assertion rank($Ric(g(0))) = k$ will follow from an argument of Yu \cite{[Yu]}.
\begin{remark}
There is a general philosophy that the Ricci flow makes the curvature towards positive, e.g., Hamilton-Ivey pinching estimate \cite{[H3]}\cite{[I]}.  So it might be interesting to see that in our case, at least in a short time, the Ricci curvature remains nonpositive.
\end{remark}

\section*{Acknowledgements}
The author would like to express his deep gratitude to his former advisor, Professor Jiaping Wang, for his kind help and useful suggestions.
He also thanks Professor Fangyang Zheng for his interest in this note.   Special thanks also go to Guoyi Xu, Bo Yang and Yuan Yuan for their helpful comments.

\section{\bf{Preliminaries}}
 \textbf{Hamilton's Maximum Principle}

Let $M^n$ be a closed oriented manifold with a smooth family of Riemannian metrics $g(t)$, $t \in [0, T]$. Let $V\rightarrow M$ be a real
vector bundle with a time dependent metric $h$ and $\Gamma(V)$ be the vector space of
smooth sections on $V$. Let $\nabla^L_t$ denote the corresponding Levi-Civita connection on $(M, g(t))$. Furthermore,
let $\nabla_t$ denote a time dependent metric connection on $V$.  For a section $R\in \Gamma(V)$, define a new section
$\Delta_t R\in \Gamma(V)$ as follows. For $p\in M$ choose an orthonormal basis of $V_p$ (the fiber of $V$ at $p$) and
extend it along the radial geodesics in $(M, g(t))$ emanating from $p$ by parallel transport of $\nabla_t$ to an orthonormal basis
$X_1(q)$, ..., $X_d(q)$ of $V_q$ for all $q$ in a small neighborhood of $p$. If $f_i$ satisfies $R = \sum\limits_{i=1}^d f_iX_i$,
then $$(\Delta_tR)(p) = \sum\limits_{i=1}^{d}(\Delta_tf_i)X_i(p)$$ where $\Delta_t$ is the Beltrami Laplacian on functions.

Suppose that a time dependent section $R(\cdot, t)\in \Gamma(V)$ satisfies the parabolic equation
\begin{equation}\label{1}
\frac{\partial R(p, t)}{\partial t} = (\Delta_tR)(p, t) + f(R(p, t))
\end{equation}
where $f: V\rightarrow V$ is a local Lipschitz map mapping each fibre $V_q$ to itself.
Roughly speaking, Hamilton's maximum principle says that the dynamics of the parabolic equation (\ref{1})
is controlled by the ordinary differential equation
\begin{equation}\label{2}
\frac{d R}{d t} = f(R(p, t)).
\end{equation}

More precisely, we have the following version of Hamilton's maximum principle \cite{[BW]}\cite{[CL]}:
\begin{theorem}\label{thm2}
For $t\in [0, \delta]$, let $C(t)\subseteq V$ be a closed subset, depending continuously on $t$. Suppose that each of the sets $C(t)$ is invariant
under parallel transport, fiberwise convex and that the family of $C(t)$ $(0\leq t\leq \delta)$ is invariant under the ordinary differential equation (\ref{2}).  Then for any solution $R(p, t)\in \Gamma(V)$ on $M\times [0, \delta]$ of parabolic equation (\ref{1}) with $R(\cdot, 0)\in C(0)$, we have $R(\cdot, t)\in C(t)$ for all $t\in [0, \delta]$.
\end{theorem}

\textbf{The Ricci flow deformation}
Let $M^n$ be a compact K\"ahler manifold. We consider an abstract complex vector bundle $W$ isomorphic to $T_{\mathbb{C}}M$ and endow it with a fixed Hermitian fiber metric $k$. Let $g_t$ denote the solution to the unnormalized Ricci flow
$$\frac{\partial}{\partial t}g_{i
\overline{j}} = -Ric_{i\overline{j}}$$
with initial metric $g_0$. We choose an isometry $u: W\to T_\mathbb{C}M$ at $t=0$ and let the isometry evolve by the equation $$\frac{\partial}{\partial t}u^i_a=\frac{1}{2}g^{i\overline{j}}Ric_{k\overline{j}}u^k_a.$$ Then, the pull-back metric on $W$ is constant in time \cite{[H2]}. With the isometry $u(t): (W, k) \to (T_{\mathbb{C}}M, g_t)$, we can pull back any complex vector bundle 
over $M$ associated to the principle bundle $P$ of unitary frames of $(T_\mathbb{C}M, g_t)$. 

Let $H$ be an $2n$ dimensional real vector space with complex structure $J$ and $J$-invariant inner product $g(\cdot, \cdot)$. Define $H_\mathbb{C} = H\otimes_{\mathbb{R}}\mathbb{C}$. Then we extend $J$ and $g$ linearly to $H_\mathbb{C}$. Define $U = \{X-\sqrt{-1}JX, X\in H\}$, $\overline{U} = \{X+\sqrt{-1}JX, X\in H\}$. Then $U\oplus\overline{U} = H_\mathbb{C}$. Let $K$ be a subspace of $\otimes^4 H^*$ such that for any $R\in K$,
\begin{itemize}
\item $R(x, y, z, w) = -R(y, x, z, w) = R(z, w, x, y)$ for $x, y, z, w\in H$;
\item $R(x, y, z, w)+R(y, z, x, w)+R(z, x, y, w) = 0$ for $x, y, z, w\in H$;
\item $R(X, Y, z, w) = 0$ for $X, Y\in U$, $z, w\in H$ (here we extend $R$ complex linearly). 
\end{itemize}

The curvature tensor $R$ of the K\"ahler manifold $(M, g)$ can be considered as a section in the associated bundle
\begin{equation}\label{-11}V = P_W\times_{U(n)}K.\end{equation} Here $P_W$ is the unitary frame bundle.
The fiber metric $k$ on $W$ induces a metric on $V$ again denoted by $k$.
We can pull back the Levi-Civita connection of the tangent bundle of $(M, g(t))$ and induce connection on any associated vector bundle. As explained in the previous section, the Laplacian of a section of such a bundle could be defined.
One arrives at the following evolution equation for the curvature tensor $R$ on $M$ (\cite{[H2]}):
\begin{equation}\label{eq11}
\begin{aligned}
\frac{\partial}{\partial t}R_{i\overline jk\overline l}&=\Delta R_{i\overline jk\overline l}+R_{i\overline jp\overline q}R_{q\overline pk\overline l}-R_{i\overline pk\overline q}R_{p\overline jq\overline l}+R_{i\overline lp\overline q}R_{q\overline pk\overline j}\\&-\frac{1}{2}(R_{i\overline p}R_{p\overline jk\overline l}+R_{p\overline j}R_{i\overline pk\overline l}+R_{k\overline p}R_{i\overline jp\overline l}+R_{p\overline l}R_{i\overline jk\overline p}),
\end{aligned}
\end{equation}
\begin{equation}\label{eq-50}
\frac{\partial}{\partial t}R_{i\overline j} = \Delta R_{i\overline j}+R_{i\overline jp\overline q}R_{q\overline p}-R_{i\overline p}R_{p\overline j}.
\end{equation} 
The corresponding ordinary equations are
\begin{equation}\label{eq10}
\begin{aligned}
\frac{d}{d t}R_{i\overline jk\overline l}&=R_{i\overline jp\overline q}R_{q\overline pk\overline l}-R_{i\overline pk\overline q}R_{p\overline jq\overline l}+R_{i\overline lp\overline q}R_{q\overline pk\overline j}\\&-\frac{1}{2}(R_{i\overline p}R_{p\overline jk\overline l}+R_{p\overline j}R_{i\overline pk\overline l}+R_{k\overline p}R_{i\overline jp\overline l}+R_{p\overline l}R_{i\overline jk\overline p}),
\end{aligned}
\end{equation}\begin{equation}\label{eq15}
\frac{d}{dt}R_{i\overline j} = R_{i\overline jp\overline q}R_{q\overline p}-R_{i\overline p}R_{p\overline j}.
\end{equation}

\medskip

\textbf{Ricci kernel foliation} This part is based on \cite{[WZ]}, page 266.
Let $M$ be a compact K\"ahler manifold with nonpositive bisectional curvature. Then the Ricci form is negative semi-definite. Let $r$ be the maximal (complex) rank of the Ricci form and $U$ be the open set on $M$ where the rank of Ric is equal to $r$. Denote by $\mathcal{L}$ the distribution in $U$ given by the kernel of the Ricci tensor.

By linear algebra, the nonpositivity of the bisectional curvature implies that $X\in\mathcal{L}$ if and only if $R(X, *, *, *) \equiv 0$, where $R$ is the curvature tensor. So $\mathcal{L}$ is the kernel of the curvature tensor. Thus it is a foliation, whose leaves are totally geodesic, flat complex submanifolds of $U$. By a theorem of Ferus \cite{[F]},  each leaf of $\mathcal{L}$ is complete.

\section{\bf{The proof of theorem \ref{thm1}}}
\begin{proof}

Let $g(t)$ be the solution to the Ricci flow equation $\frac{\partial g_{i\overline j}}{\partial t} = -Ric_{i\overline j}$ with $g(0) = g$. Let $V$ be defined as in (\ref{-11}).  Following B\"ohm and Wilking  \cite{[BW]}, we shall construct a family of convex sets $V_t$ of $V$ which are invariant under parallel transport and $V_t$ is invariant under the evolution equation (\ref{eq10}) for small $t$. The precise statement is the following:
\begin{prop}\label{prop}
Let $K_2$ be a positive constant. Then there exist positive constants $K_1, K_3, \epsilon$ depending only on $K_2$ and $n$  such that $\epsilon K_1\leq 1$, $\epsilon K_3\leq \frac{1}{2}$ and the following hold:
For $t\geq 0$, let $V_t$ be a subset of K\"ahler curvature operators $R$ in $V$ satisfying the following conditions:

 (1). $Ric(\alpha, \overline{\alpha}) \leq 0$ for any $e_\alpha \in W^{1, 0}$.

 (2). $|R_{x\overline xu\overline v}|^2 \leq (1+tK_1)Ric(u, \overline u)Ric(v, \overline v)$ for any $x , u , v \in W^{1, 0}$ and $|x|= 1$.

 (3).$||R|| \leq K_2 + tK_3$($||\cdot||$ is the norm with respect to the bundle metric).
 
 Then $V_t$ is closed, fiberwise convex, invariant under parallel transport for each $t$. Moreover, $V_t$ is invariant under the evolution equation (\ref{eq10}) for $0 \leq t < \epsilon$. 
\end{prop}
\begin{proof}
Through the course of the proof of proposition \ref{prop}, we will assume $\epsilon K_1\leq 1$. The explicit values of $\epsilon$ and $K_1$ will be determined by the end of the proof of the proposition.
It is clear that $V_t$ is closed and invariant under parallel transport. Given any positive constants $K_2, K_1, K_3, \epsilon$, we prove $V_t$ is fiberwise convex for each $t$.
It is easy to see that (3) defines a convex set.
Suppose $R$ and $S$ in $V$ satisfy (1) and (2). For any $0 \leq \lambda \leq 1$, 
define $$T = \lambda R + (1-\lambda)S.$$  Obviously $T$ satisfies (1). For any $x , u , v \in W^{1, 0}$ and $|x|= 1$,
\begin{equation}
\begin{aligned}
|T_{x\overline xu\overline v}|^2 &= |\lambda R_{x\overline xu\overline v} + (1-\lambda)S_{x\overline xu\overline v}|^2 \\&\leq (1 + tK_1)| \lambda \sqrt{Ric_R(u, \overline u)Ric_R(v, \overline v)} +  (1-\lambda)\sqrt{Ric_S(u, \overline u)Ric_S(v, \overline v)}|^2 \\&\leq (1 + tK_1)(\lambda Ric_R(u, \overline u) + (1-\lambda) Ric_S(u, \overline u))(\lambda Ric_R(v, \overline v) + (1-\lambda) Ric_S(v, \overline v))\\&=(1+tK_1)Ric_T(u, \overline{u})Ric_T(v, \overline{v}).
\end{aligned}
\end{equation}
Therefore, $V_t$ is fiberwise convex.  

\bigskip

Next we show that $V_t$ is invariant under (\ref{eq10}) for $0\leq t\leq \epsilon$. For any $a\in [0, \epsilon)$,
let $R^{\lambda}$ be the solution to (\ref{eq10}) with the initial condition $R^\lambda(a) = R(a) - \lambda R'$.
 Here $R(a)\in V_a$; $\lambda$ is a small positive number; $R'_{i\overline jk\overline l}= c(n)(\delta_{ij}\delta_{kl}+\delta_{il}\delta_{jk})$, where $c(n)$ is a constant such that $||R'|| = 1$. 
 
 \begin{lemma} There exist positive constants $\epsilon, A, K_1, K_3$ depending only on $K_2$ and $n$ (independent of $\lambda$) such that $\epsilon K_1\leq 1$, $\epsilon K_3\leq \frac{1}{2}$ and for any $t \in [a, \epsilon]$,
the solution $R^{\lambda}$ satisfies

(1'). $Ric^{\lambda}(\alpha, \overline{\alpha}) \leq -\frac{\lambda}{2n}e^{-At}$ for any $e_\alpha \in W^{1, 0}$ and $|e_\alpha| = 1$.

(2'). $|R^{\lambda}_{x\overline xu\overline v}|^2 \leq (1+tK_1)Ric^{\lambda}(u, \overline u)Ric^{\lambda}(v, \overline v)$ for any $x , u , v \in W^{1, 0}$ and $|x| = 1$.

(3'). $||R^{\lambda}|| \leq K_2 +\lambda+ tK_3$.
\end{lemma}

\begin{proof}
We may assume $\lambda\leq \frac{1}{2}$.  Since $R(a)\in V_a$ and $\epsilon K_3\leq \frac{1}{2}$ (this will be verified very soon), $$||R(a)||\leq K_2+aK_3\leq K_2+\frac{1}{2}.$$ $$||R^\lambda(a)||\leq ||R(a)||+\lambda||R'||\leq K_2+aK_3+\lambda\leq K_2+1.$$ By (\ref{eq10}), $\frac{d ||R^{\lambda}||}{dt}\leq C(n)||R^{\lambda}||^2$ where $C(n)$ is a constant depending only on $n=dim_{\mathbb{C}}(M)$. We find that if $t \leq \frac{1}{2C(n)(K_2+1)}$, 

\begin{equation}\label{eq13}||R^\lambda(t)||\leq 3(K_2+1), \frac{d ||R^{\lambda}||}{dt}\leq C(n)||R^{\lambda}||^2\leq 9C(n)(K_2+1)^2.\end{equation}

  Take \begin{equation}\label{eq12}K_3=10C(n)(K_2+1)^2, \epsilon\leq \frac{1}{20C(n)(K_2+1)^2}.\end{equation} Then (\ref{eq13}) is valid; $R^\lambda(t)$ satisfies (3') for $t\leq \epsilon$; $\epsilon K_3\leq \frac{1}{2}$. The explicit value of $\epsilon$ will be determined later.

\begin{claim}\label{cl1}
If $R^{\lambda}$ satisfies the Lemma at some time $t\leq  \epsilon$, then there exists $C > 0$ depending only on $K_2$ and $n$ such that at time $t$, $|R^{\lambda}_{i\overline{j}k\overline{l}}| \leq C\sqrt{-Ric^{\lambda}(i, \overline{i})}$ and $|R^{\lambda}_{i\overline{j}k\overline{l}}| \leq C\sqrt{Ric^{\lambda}(i, \overline{i})Ric^{\lambda}(j, \overline{j})}$ for any $e_i, e_j, e_k, e_j \in W^{1, 0}$ with length $1$.
\end{claim}
\begin{proof}
(1') implies $Ric^{\lambda}\leq 0$; Since $\epsilon K_1\leq 1$, (2') says $|R^{\lambda}_{x\overline xu\overline v}|^2 \leq 2Ric^{\lambda}(u, \overline u)Ric^{\lambda}(v, \overline v)$ for any $x , u , v \in W^{1, 0}$ and $|x| = 1$.

We polarize the curvature operator. 
\begin{equation}\label{eq14}
\begin{aligned}
R^\lambda_{i\overline jk\overline l} &= \frac{1}{4}(R^\lambda(e_i, \overline{e_j}, e_k+e_l, \overline{e_k}+\overline{e_l}) - R^\lambda(e_i, \overline{e_j}, e_k-e_l, \overline{e_k}-\overline{e_l})\\&+\sqrt{-1}R^\lambda(e_i, \overline{e_j}, e_k+\sqrt{-1}e_l, \overline{e_k}-\sqrt{-1}\overline{e_l})\\&-\sqrt{-1}R^\lambda(e_i, \overline{e_j}, e_k-\sqrt{-1}e_l, \overline{e_k}+\sqrt{-1}\overline{e_l})).
\end{aligned}
\end{equation}

Each term is bounded by the Ricci curvature. For instance, 
\begin{equation}
\begin{aligned}
|R^\lambda(e_i, \overline{e_j}, e_k+e_l, \overline{e_k}+\overline{e_l}) |&\leq 2\sqrt{Ric^\lambda(i, \overline i)Ric^\lambda(j, \overline j)}|e_k+e_l|^2\\& \leq 8\sqrt{Ric^\lambda(i, \overline i)Ric^\lambda(j, \overline j)}\\&\leq C'\sqrt{(K_2+1)}\sqrt{-Ric^\lambda(i, \overline i)}.
\end{aligned}
\end{equation}
In the last inequality, we used (\ref{eq13}). Here $C'$ is a constant depending only on $n$.
Similarly, other three terms in the right hand side of (\ref{eq14}) could be bounded. The proof of Claim \ref{cl1} is complete.
\end{proof}

Below $C_s(s= 1, 2,...)$ are positive constants depending only on $n$ and $K_2$.
It is easy to see that (1'), (2') and (3') in the Lemma hold for $t = a$.
If the Lemma is not true, let $t_0 = \sup\{b | $the Lemma holds for $a\leq t\leq b\}<\epsilon$.  
Therefore, Claim \ref{cl1} holds at $t=t_0$.
There are only two possibilities:

(i) (1') does not hold on $[a, t_1)$ for any $t_1 > t_0$.

\medskip

(ii) (2') does not hold on $[a, t_1)$ for any $t_1 > t_0$.

\medskip

In case (i), Let $e_i (i = 1, 2,..., n)$ be a unitary frame in $W^{1, 0}$. For any $e_\alpha \in W^{1, 0}$ with $|e_\alpha| = 1$, (\ref{eq15}), (\ref{eq13}) and Claim \ref{cl1}  imply
\begin{equation}\label{eq5}
\begin{aligned}
\frac{d}{dt}Ric^\lambda(\alpha, \overline\alpha)|_{t=t_0}&=\sum\limits_{i, j}R^\lambda_{\alpha\overline\alpha i\overline j}R^\lambda_{j\overline i}-R^\lambda_{\alpha\overline i}R^\lambda_{i\overline\alpha}\\& \leq C_3\sqrt{Ric^\lambda(\alpha, \overline\alpha)Ric^\lambda(\alpha, \overline\alpha)}\cdot C_6+C_4\sqrt{-Ric^\lambda(\alpha, \overline\alpha)}\cdot C_5\sqrt{-Ric^\lambda(\alpha, \overline\alpha)}
\\&= -C_1Ric^\lambda(\alpha, \overline\alpha).
\end{aligned}
\end{equation}  By (1') and our assumption, at $t=t_0$, $$Ric^\lambda(\alpha, \overline\alpha)\leq -\frac{\lambda}{2n}e^{-At_0}.$$ Take
\begin{equation}\label{eq16}A = 2C_1.\end{equation} (\ref{eq5}) implies that $Ric^\lambda(\alpha, \overline\alpha)\leq -\frac{\lambda}{2n}e^{-At}$ for $0\leq t\leq t_0+\delta$ with some $\delta>0$. This contradicts (i).
 
\medskip

For case (ii), let $x\in W^{1, 0}$ with $|x| = 1$. Note that
\begin{equation}\label{eq19}
t_0K_1< \epsilon K_1 \leq 1.\end{equation}
By a computation similar to (\ref{eq5}),
\begin{equation}\label{eq18}
\begin{aligned}
&\frac{d}{dt} ((1+tK_1)Ric^\lambda(u, \overline u)Ric^\lambda(v, \overline v) - |R^\lambda_{x\overline xu\overline v}|^2)|_{t=t_0}\\&=K_1Ric^\lambda(u, \overline u)Ric^\lambda(v, \overline v)-(1+t_0K_1)\frac{d}{dt}(Ric^\lambda(u, \overline u)Ric^\lambda(v, \overline v))-\frac{d}{dt}|R^\lambda_{x\overline xu\overline v}|^2\\&\geq K_1Ric^\lambda(u, \overline u)Ric^\lambda(v, \overline v)-2(|\frac{d}{dt}Ric^\lambda(u, \overline u)||Ric^\lambda(v, \overline v)|+|\frac{d}{dt}Ric^\lambda(v, \overline v)||Ric^\lambda(u, \overline u)|)\\&-2||R^\lambda_{x\overline xu\overline v}||\frac{d}{dt}R^\lambda_{x\overline xu\overline v}|\\&\geq
(K_1-C_2)Ric^\lambda(u, \overline u)Ric^\lambda(v, \overline v).
\end{aligned}
\end{equation}
In the inequality, Claim \ref{cl1}, (\ref{eq10}), (\ref{eq13}) and (\ref{eq19}) are applied.
 By (2') and our assumption, $$(1+tK_1)Ric(u, \overline u)Ric(v, \overline v) - |R_{x\overline xu\overline v}|^2\geq 0$$ at $t=t_0$. 
Take \begin{equation}\label{eq17}K_1 = 2C_2+10,  \epsilon =min(\frac{1}{2(2C_2+10)},  \frac{1}{20C(n)(K_2+1)^2}).\end{equation}
Then (\ref{eq12}) and (\ref{eq19}) are valid. Therefore (\ref{eq18}) holds. Moreover, if $u, v\neq 0$, $(\ref{eq18})> 0$ by (1'). This means (ii) cannot happen for $t_0< \epsilon$. Putting (\ref{eq12}), (\ref{eq16}), (\ref{eq17}) and (\ref{eq19}) together, 
we prove the lemma.
\end{proof}
Proposition \ref{prop} follows if we let $\lambda \to 0$ in the Lemma.
\end{proof}

\bigskip

Take $K_2 = 2||R||$ in proposition \ref{prop}, where $R$ is the curvature tensor of $(M^n, g)$ and $||\cdot||$ is the $C^0$ norm with respect to $g_0=g$.
\begin{claim}\label{cl0}$R\in V_0$ where $V_0$ is defined in proposition \ref{prop}.
\end{claim}
\begin{proof}
  (1) and (3) are automatic, since $(M, g)$ has nonpositive bisectional curvature.
To check (2), we notice that for fixed $x$, $R_{x\overline{x}p\overline{q}}$ is a Hermitian form. Let $e_i$ be the eigenvectors for $i = 1, 2,.., n$ and $$R_{x\overline{x}e_i\overline{e_j}} = \delta_{ij}\lambda_i,$$ where $\lambda_i$ are all nonpositive. Suppose $u = \sum\limits_{i=1}^{n}u_ie_i, v = \sum\limits_{i=1}^{n}v_ie_i$, then \begin{equation}
\begin{aligned}
|R_{x\overline{x}u\overline{v}}|^2 &= |\sum\limits_{i=1}^{n}u_i\overline{v_i}\lambda_i|^2 \\&\leq (\sum\limits_{i=1}^{n}|u_i\sqrt{-\lambda_i}|^2)(\sum\limits_{i=1}^{n}|\overline{v_i}\sqrt{-\lambda_i}|^2) \\&= R_{x\overline{x}u\overline{u}}R_{x\overline{x}v\overline{v}} \\&\leq Ric(u, \overline{u})Ric(v, \overline{v}).
\end{aligned}
\end{equation}
\end{proof}

\bigskip

Putting Claim \ref{cl0}, proposition \ref{prop} and theorem \ref{thm2} together, we find 
\begin{theorem}\label{thm-2}
 Let $(M, g)$ be compact K\"ahler manifold with nonpositive bisectional curvature. If $g_t$ satisfies the Ricci flow equation $\frac{\partial g_t}{\partial t} = -Ric(g_t)$ and $g_0 = g$, then there exists $\epsilon>0$ depending only on the bound of the curvature and the dimension such that $Ric(g_t)\leq 0$ for $0\leq t< \epsilon$.
\end{theorem}
\begin{remark}
The counterpart of theorem \ref{thm-2} is true in the Riemannian case, i.e., if a compact manifold has nonpositive sectional curvature, then along the Ricci flow, in a short time, the Ricci curvature remains nonpositive.
\end{remark}

Now let us come back to theorem \ref{thm1}. By theorem \ref{thm-2}, $Ric(g(t)) \leq 0$ for small $t>0$. Following the arguments in \cite{[BW]} (page 676-677), we shall show that the rank of $Ric_t$ is constant and the null space is parallel. 
By (\ref{eq-50}),
$$\frac{\partial}{\partial t}Ric_{v\overline v} = \Delta Ric_{v\overline v}+\sum\limits_{p, q}R_{v\overline vp\overline q}Ric_{q\overline p}-\sum\limits_{p}Ric_{v\overline p}Ric_{p\overline v}.$$
Let $H$ be a constant and define $\tilde{Ric}_t = e^{Ht} Ric_t$. By proposition \ref{prop} and Claim \ref{cl1}, if $H$ is large (depending on $n$ and $K_2$), 
\begin{equation}\label{4}
\begin{aligned}
\frac{\partial\tilde{Ric}_{v\overline v}}{\partial t} &=He^{Ht}Ric_{v\overline v}+e^{Ht}\frac{\partial Ric_{v\overline v}}{\partial t}\\&
=e^{Ht}(HRic_{v\overline v}+\sum\limits_{p, q}R_{v\overline vp\overline q}Ric_{q\overline p}-\sum\limits_{p}Ric_{v\overline p}Ric_{p\overline v})+\Delta_t\tilde{Ric}_{v\overline v}\\&\leq \Delta_t \tilde{Ric}_{v\overline v}.
\end{aligned}
\end{equation}
 If $Ric < 0$ for some small $t > 0$, then $c_1(M) < 0$. Otherwise, the rank of the Ricci tensor is less
than $n$ for all  $\epsilon>t > 0$. 
Let $0\geq \mu_1\geq \mu_2\geq....\geq\mu_n$ denote the eigenvalues of $\tilde{Ric}$ and let
$$\sigma_l=\mu_1+\mu_2+....+\mu_l.$$ Fix $p\in M$ and let $e_1(t_0), e_2(t_0),..., e_l(t_0)$ be an orthogonal basis of $T^{1, 0}_p(M)$ such that
$\sigma_l(t_0) = \sum\limits_{i=1}^{l}\tilde{Ric}_{t_0}(e_i(t_0), \overline{e_i(t_0)})$.
\begin{equation}
\begin{aligned}
\sigma'_l(t_0): &= \lim\limits_{t\nearrow t_0}\sup \frac{\sigma_l(t_0)-\sigma_l(t)}{t_0-t}\\&
\leq\frac{d}{dt}|_{t=t_0}\sum\limits_{i=1}^l\tilde{Ric}_t(e_i(t_0), \overline{e_i(t_0)})\\&\leq \sum\limits_{i=1}^l\Delta\tilde{Ric}_{t_0}(e_i(t_0), \overline{e_i(t_0)})
\\&\leq\Delta\sigma_l\end{aligned}
\end{equation}
Thus $$\frac{\partial \sigma_l}{\partial t} \leq \Delta\sigma_l$$
in the support function sense. By the strong maximum principle, for some $\epsilon_1 > 0$, either $\sigma_l < 0$ or $\sigma_l \equiv 0$ for $t\in(0, \epsilon_1]$. Therefore we can assume that the rank of $Ric_t$ is a constant $k$ for $0<t\leq\epsilon_1$.

For any point $p\in M$, let $U$ be a small neighborhood containing $p$. 
Consider a smooth vector field $v(t) \in T^{1, 0}U$ for $0<t<\epsilon_1$ depending smoothly on $t$ such that $\tilde{Ric}_t(v, \overline v) = 0$.
Since $\tilde{Ric} \leq 0$,  

\begin{equation}\label{eq21}\tilde{Ric}(v, \overline s)=\tilde{Ric}(s, \overline v) = 0\end{equation} for any $s\in T^{1, 0}M$. Let $e_i\in T^{1, 0}M$ be a local unitary frame on $M$ and $s$ be a smooth section of $T^{1,0}M$.
Then $0 = e_i(\tilde{Ric}(v, \overline s)) = (\nabla_{e_i}\tilde{Ric})(v, \overline s)+\tilde{Ric}(\nabla_{e_i}v, \overline s)$. This means
\begin{equation}\label{eq22}(\nabla_{e_i}\tilde{Ric})(v, \overline s)=-\tilde{Ric}(\nabla_{e_i}v, \overline s);
(\nabla_{\overline{e_i}}\tilde{Ric})(s, \overline v)=-\tilde{Ric}(s, \nabla_{\overline{e_i}}\overline v).
\end{equation}

By (\ref{4}), (\ref{eq21}) and (\ref{eq22}), $$\begin{aligned}0 &= \frac{\partial}{\partial t}(\tilde{Ric}(v, \overline v))\\&=(\frac{\partial}{\partial t}\tilde{Ric})(v, \overline v)+ \tilde{Ric}(\frac{dv}{dt}, \overline v) + \tilde{Ric}(v, \frac{d\overline v}{dt})  \\&= (\frac{\partial}{\partial t}\tilde{Ric})(v, \overline v)\\&\leq(\Delta \tilde{Ric})(v, \overline v)\\& = (\sum\limits_i(\nabla_{e_i}\nabla_{\overline{e_i}}+\nabla_{\overline{e_i}}\nabla_{e_i})\tilde{Ric})(v, \overline v)\\&=\Delta(\tilde{Ric}(v, \overline v)) -\sum\limits_i(2\tilde{Ric}(\nabla_{e_i}v, \overline{\nabla_{e_i}v})+2(\nabla_{e_i}\tilde{Ric})(v, \nabla_{\overline i}\overline v)+2(\nabla_{\overline{e_i}}\tilde{Ric})(\nabla_{e_i}v, \overline v))\\&
=2\tilde{Ric}(\nabla_{e_i}v, \overline{\nabla_{e_i}v})\end{aligned}$$ 

This shows that the kernel
of $Ric_t$ is parallel for $0<t<\epsilon_1$.  Consider
\begin{equation}\label{eq99}
0 = \frac{\partial}{\partial t}(Ric(v, \frac{d\overline v}{dt})) = (\frac{\partial}{\partial t}Ric)(v, \frac{d\overline v}{dt})+Ric(\frac{dv}{dt}, \frac{d\overline v}{dt}).
\end{equation}
Let $s\in T^{1, 0}M$. Then
 \begin{equation}\label{eq98}
 \begin{aligned}
 (\frac{\partial}{\partial t}Ric)(v, \overline s)& =  \Delta Ric_{v\overline s}+\sum\limits_{p, q}R_{v\overline sp\overline q}Ric_{q\overline p}-\sum\limits_{p}Ric_{v\overline p}Ric_{p\overline s}\\&= \Delta Ric_{v\overline s}\\&=0.
\end{aligned}
\end{equation}
In the second equality, we used (2) in proposition \ref{prop}. In the last step, we used that the kernel of $Ric_t$ is parallel.
Take $s = \frac{dv}{dt}$. Then (\ref{eq99}) and (\ref{eq98}) imply $Ric_t(\frac{dv}{dt}, \frac{d\overline v}{dt}) = 0$. This means the kernel of $Ric_t$ is invariant for $0<t<\epsilon_1$.

By proposition \ref{prop},  the kernel of the Ricci tensor is the kernel of the curvature operator. De Rham theorem says the universal cover ($\tilde M, g(t)$) has a flat factor $\mathbb{C}^{n-k}$ for $0<t<\epsilon_1$. By a holonomy argument, the universal cover $(\tilde{M}, g_0)$ is biholomorphic and isometric to $\mathbb{C}^{n-k} \times Y^{k}$ where $Y^k$ is a complete K\"ahler manifold with nonpositive bisectional curvature. 

\medskip
Let $r$ be the maximal rank of the Ricci curvature of $g=g_0$. We follow the argument in \cite{[Yu]} to show that $r=k=dim(Y)$. It is clear that \begin{equation}\label{eq-3}r\leq k,\end{equation} since the rank of $Ric_t$ is $k$ for $0<t\leq\epsilon_1$.
Recall corollary  C in \cite{[WZ]} (page 277):
\begin{theorem}\label{thm4}
If $M^n$ is a compact K\"ahler manifold with nonpositive
bisectional curvature which has Ricci rank $r < n$, then the open set U
in which the Ricci tensor has maximum rank $r$ in the universal cover
 $\tilde{M}$ is, locally, holomorphically isometric to $L_a\times Y_a$, where $L_a$
is a complete flat K\"ahler manifold, and $Y_a$ is a K\"ahler manifold with
nonpositive bisectional curvature and negative Ricci curvature.
\end{theorem}

We apply theorem \ref{thm4} to the compact K\"ahler manifold $M$ in theorem \ref{thm1}. It is immediate to see that $dim(Y_a) = r$ since $Y_a$ has negative Ricci curvature.
Let $f$ be the holomorphic immersion $L_a \to \tilde{M}$ given by theorem \ref{thm4}.
By the evolution equation of the K\"ahler-Ricci flow, 
\begin{equation}\label{5}
\frac{\partial}{\partial t}Ric = \sqrt{-1}\partial\overline\partial R
\end{equation}
 where $R$ is the scalar curvature and $Ric = \sqrt{-1}R_{i\overline{j}}dz^i\wedge dz^{\overline{j}}$.
 Let $p$ be any point in $f(L_a)$. 
 For $e_i\in T^{1,0}_pf(L_a)$,  by theorem \ref{thm-2}, $Ric(g(\epsilon_1)) \leq 0$ and $Ric_{i\overline{i}}(g(0)) = 0$. Therefore \begin{equation}\label{6}
 0 \geq f^*Ric_{i\overline{i}}(g(\epsilon_1))-f^*Ric_{i\overline{i}}(g(0)) = \sqrt{-1}\partial_i\partial_{\overline{i}}\int\limits_0^{\epsilon_1} R(p, t)dt.
 \end{equation} 
 $(\ref{6})$ implies that $-\int\limits_0^{\epsilon_1} R(p, t)dt$ is a bounded plurisubharmonic function on $L_a$. Since $L_a$ is flat, the function must be a constant. Therefore $Ric_{i\overline i}(g(\epsilon_1)) = 0$ for any
 $e_i\in T^{1,0}L_a$. Then $k = $ rank of $Ric(g(\epsilon_1)) \leq n- $dim($L_a$)$ = n-(n-r) = r$. Combining this with (\ref{eq-3}), we find $r = k$. Therefore, for the metric $g=g(0)$, the Ricci curvature is negative somewhere on $Y$.

\medskip
Recall theorem $E$ in \cite{[WZ]} (page 278):
\begin{theorem}
Let $M^n$ be a compact K\"ahler manifold with real analytic metric and nonpositive bisectional curvature. Denote by $r$ its Ricci rank. Then there exists a finite covering $M'$ of $M$, such that $q: M'\to N^r$ is a holomorphic fiber bundle over compact K\"ahler manifold $N^r$ with nonpositive bisectional curvature and $c_1(N)<0$, while the fiber of $q$ is a 
complex $(n-r)$ torus $T$. 

Furthermore, $M'$ is diffeomorphic to $N\times T$, and $q$ is a metric bundle, i.e., $\forall x\in N$, there exists  a small neighborhood $x\in V\subset N$ such that $q^{-1}V$ is isometric to $T\times V$.
 
\end{theorem}

In Wu and Zheng's proof of theorem $E$, the real analyticity condition is only used to show that the universal cover splits as $\mathbb{C}^{n-r}\times Y^r$ where $Y^r$ is a simply connected, complete K\"ahler manifold with nonpositive bisectional curvature, and the Ricci tensor of $Y^r$ is negative definite somewhere. 
Since this is confirmed without assuming the real analyticity of the metric, Wu and Zheng's proof works in our case without any modification.

Next we show $r = k =Kod(M)$.
Recall theorem $6.10(2)$ in \cite{[U]}:
\begin{theorem}
Let $f: V\to W$ be a finite unramified covering of complex manifolds. Then $Kod(V) = Kod(W)$.
\end{theorem}
Since $M'$ is a finite cover of $M$, $Kod(M) = Kod(M')$. Note that pluricanonical sections on $M'$ could be reduced to pluricanonical sections on $N$.  Thus $Kod(M') = Kod(N) = k$, since $c_1(N)<0$. The proof of Theorem \ref{thm1} is complete.
\end{proof}

\section{\bf{The proof of theorem \ref{thm5}}}
\begin{proof}
Let $g(t)$ be the solution to the K\"ahler-Ricci flow with $g(0)=g$. Then by theorem \ref{thm-2}, the Ricci curvature will be nonpositive in a short time. 
By assumption, $N$ is an immersed totally geodesic flat complex submanifold of $M$ and $Ric(M)|_{TN} = 0$. Applying equation $(\ref{6})$, we find that $Ric(M, g(t))|_{TN}$ vanishes for small $t>0$. Then rank$(Ric(g(0))$ = rank$(Ric(g(t))\leq k$.  By corollary \ref{cor-1},  $\tilde{M}$ has a flat factor $\mathbb{C}^{n-k}$.
\end{proof}

Proof of corollary $3$:
Let $r$ be the maximal rank of the Ricci tensor of $M$. We only need to prove that  if the leaf $L^{n-i}$ through $p$ stays in the interior of $U(i)$, then $i = r$.
In this case, $L^{n-i}$ must be a complete totally geodesic immersed complex submanifold in $M$ such that $Ric(M)|_{TL} = 0$. 
By theorem $\ref{thm5}$, the universal cover of $M$ splits off a factor $\mathbb{C}^{n-i}$. Thus $r \leq n-(n-i) = i$.
The proof of corollary $3$ is complete.

\end{document}